\documentclass[12pt]{amsart}

\usepackage{amssymb}
\usepackage{fullpage}
\usepackage{amsmath}
\usepackage{amsfonts}
\usepackage{latexsym}
\usepackage{hyperref}

\usepackage{amsthm}
\usepackage{amscd}
\usepackage{verbatim}
\usepackage{setspace}
\usepackage{graphicx} 
\usepackage{url}
\usepackage{cite}

\usepackage{enumitem}
\setlist{leftmargin=20mm}

\theoremstyle{theorem}
\newtheorem{thm}{Theorem} 
\numberwithin{thm}{section}

\newtheorem{cor}[thm]{Corollary} 
\newtheorem{conj}[thm]{Conjecture}
\newtheorem{prop}[thm]{Proposition} 
\theoremstyle{definition}
 
\newtheorem{defn}[thm]{Definition}
\newtheorem{ques}[thm]{Question}
\newtheorem{nota}[thm]{Notation}

\theoremstyle{remark}

\DeclareMathOperator{\lcm}{lcm}

\title{Two Classes of Modular $p$-Stanley sequences}
\author{Mehtaab Sawhney}
\thanks{Massachusetts Institute of Technology, Cambridge MA. Email: \texttt{msawhney@math.mit.edu}}
\author{Jonathan Tidor}
\thanks{Massachusetts Institute of Technology, Cambridge MA. Email: \texttt{jtidor@mit.edu}}
\date{}

\begin{document}

\begin{abstract}
Consider a set $A$ with no $p$-term arithmetic progressions for $p$ prime. The $p$-Stanley sequence of a set $A$ is generated by greedily adding successive integers that do not create a $p$-term arithmetic progression. For $p>3$ prime, we give two distinct constructions for $p$-Stanley sequences which have a regular structure and satisfy certain conditions in order to be modular $p$-Stanley sequences, a set of particularly nice sequences defined by Moy and Rolnick which always have a regular structure.

Odlyzko and Stanley conjectured that the 3-Stanley sequence generated by $\{0,n\}$ only has a regular structure if $n=3^k$ or $n=2\cdot 3^k$. For $p>3$ we find a substantially larger class of integers $n$ such that the $p$-Stanley sequence generated from $\{0,n\}$ is a modular $p$-Stanley sequence and numerical evidence given by Moy and Rolnick suggests that these are the only $n$ for which the $p$-Stanley sequence generated by $\{0,n\}$ is a modular $p$-Stanley sequence. Our second class is a generalization of a construction of Rolnick for $p=3$ and is thematically similar to the analogous construction by Rolnick.
\end{abstract}

\maketitle
\section{Introduction}
For an odd prime $p$, a set is called \textit{$p$-free} if it contains no $p$-term arithmetic progression. Szekeres conjectured that for $p$ an odd prime, the maximum number of elements in a $p$-free subset of $\{0,1,\ldots,n-1\}$ grows as $n^{\log_{p-1}p}$ \cite{erd2}. This conjecture however has been disproved. In particular, Elkin \cite{lower} proves the best known lower bound for $3$-free sets of $O(n^{1-o(1)})$ while the best proven upper bound is $O(n(\log\log n)^5/\log n)$ due to recent work of Sanders \cite{upper}.

The inspiration for Szekeres's conjecture however is of interest. In particular, Szekeres's conjecture is based on the sequence constructed by starting with 0 and greedily adding each subsequent integer that does not create a $p$-term arithmetic progression. The sequence produced is exactly the nonnegative integers that have no digit of $p-1$ in their base $p$ expansion. In 1978, Odlyzko and Stanley generalized this construction to arbitrary sets \cite{stan}.

\begin{defn}
Let $A:=\{a_1,\ldots,a_n\}$ be a finite set of nonnegative integers that contains $0$ with no nontrivial $p$-term arithmetic progressions. Furthermore take $0=a_1<a_2<\cdots<a_n$ and for each integer $k\ge n$, let $a_{k+1}$ be the least integer greater than $a_k$ such that $\{a_1,\ldots, a_k,a_{k+1}\}$ has no $p$-term arithmetic progressions. The \textit{$p$-Stanley sequence} $S_p(A)$, also written as $S_p(a_1,\ldots, a_n)$, is the sequence $a_1,\ldots, a_n, a_{n+1}, \ldots$.
\end{defn}

In the language of Stanley sequences the previous example is precisely $S_p(0)$. Odlyzko and Stanley noticed that for some sets $A$, the Stanley sequence $S_3(A)$ displays a regular pattern in terms of the ternary representations of its terms and these sequences grow as $n^{\log_23}$. In particular, they explicitly computed $S_3(0,3^k)$ and $S_3(0,2\cdot3^k)$ and showed that these sequences satisfy the above properties. However, for other values of $m$, the sequence $S_3(0,m)$ seems to grow chaotically and at the rate $n^2/\log n$. In particular, Lindhurst \cite{L} computed $S_3(0,4)$ for large values and observes that it appears to follow this second growth rate.

Odlyzko and Stanley provided a heuristic argument why a randomly chosen sequence should grow at the rate $n^2/\log n$ and conjectured that these two behaviors are the only possible ones. Further work on the growth of chaotic $p$-Stanley sequences for $p>3$ can be found in \cite{t2}. This leads to the following conjecture, which is explicitly stated for $p=3$ in \cite{stan}.

\begin{conj}[Based on \cite{stan},\cite{t2}]
A $p$-Stanley sequence $a_1,a_2,\ldots$ with $p$ an odd prime satisfies either:
\begin{itemize}
\item[Type 1:]$a_n=\Theta(n^{\log_{(p-1)}p})$
\item[Type 2:]$a_n=\Theta\left(n^{(p-1)/(p-2)}/(\log n)^{1/(p-2)}\right)$.
\end{itemize}
\end{conj}

To date however there has been no $3$-Stanley sequence, or more generally $p$-Stanley sequence, that has been proven to have Type 2 growth. Despite this, there has been significant interest in studying the structure of Type 1 $3$-Stanley sequences (\!\!\cite{mod}, \cite{rol1}, \cite{rol2}). The most relevant class of Type 1 $3$-Stanley sequences stems from the work of Moy and Rolnick \cite{mod}, extending work of Rolnick \cite{rol2}, which gave the following class of Type 1 sequences.

\begin{defn}
Consider a set $A\subseteq\{0,\ldots,N-1\}$ with $0\in A$ such that there is no nontrivial $3$-term arithmetic progressions mod $ N$ among the elements of $A$. (Trivial arithmetic progressions refer to progressions with all elements equal.) A set $A$ is said to be \emph{modular} if for every integer $x$, there exists $y\ge z$ in $A$ such that $2y-z\equiv x\mod N$. Note that the second condition is equivalent to $x$, $y$, and $z$ being an arithmetic progression mod $ N$. Furthermore we say that $S_3(A)$ to is a \emph{modular} Stanley sequence if $A$ satisfies these conditions.
\end{defn}

Moy and Rolnick \cite{mod} conjecture that all $3$-Stanley sequences with Type 1 growth are all pseudomodular, a suitable generalization of modular sequences. In contrast, with general $p$-Stanley sequences, there is no such conjectured form for Type 1 sequences. However there is a natural analog of modular Stanley sequences, modular $p$-Stanley sequences. In particular one modifies the given definition to have no $p$-term arithmetic progressions and defines an analog of the second condition. This is defined more precisely in the next section.

In this paper we present two classes of modular $p$-Stanley sequences, one of which hints a difference between $3$-Stanley sequences and $p$-Stanley sequences for larger primes $p$ whereas the other appears to suggest a degree of similarity. The first demonstrates that for $p>3$, there exists a large class of integers $n$ for which $S_p(0,n)$ has Type 1 growth and in fact is a modular sequence. In particular for $p\ge 5$, if $2\cdot p^{k-1}<n<p^k$ and $p^k-n$ has no $p-1$ in its base $p$ expansion, then $S_p(0,n)$ has Type 1 growth. This is notable as there exist $n\neq i\cdot p^k$ for $1\le i\le p-1$ such that $S_p(0,n)$ exhibits Type 1 growth, unlike the case $p=3$ where Stanley and Odlyzko \cite{stan} conjecture only $S_3(0,3^k)$ and $S_3(0,2\cdot 3^k)$ have Type 1 growth among sequences of the form $S_3(0,n)$. Numerical evidence given by Moy and Rolnick \cite{mod} suggests that these are the only possible integer $n$ and thus appears to give a conjectural answer to a question raised by Moy and Rolnick \cite{mod} of classifying integers $n$ such that $S_p(0,n)$ is modular.

The second class is a generalization of Theorem 1.2 by Rolnick \cite{rol2}. These constructions are notable as they are among the first explicit constructions for large classes of modular $p$-sequences, with the only other large class of constructions present in the literature being that of basic sequences given by Moy and Rolnick \cite{mod}.

In Section \ref{sec:def} we provide some definitions and basic results on modular $p$-Stanley sequences that are used within this paper. In Section \ref{sec:const1} we demonstrate the first class of modular $p$-Stanley sequences, and in Section \ref{sec:const2} we demonstrate the second class of modular $p$-Stanley sequences. Section \ref{sec:concl} contains some ideas for future work in these directions.

\section{Definitions}
\label{sec:def}

This section provides the definitions and basic results on modular $p$-Stanley sequences necessary to prove our results. For further exposition, see \cite{mod}.

\begin{defn}
A set $A$ \textit{$p$-covers} $x$ if there exist $x_1,x_2,\ldots,x_{p-1}\in A$ such that $x_1<x_2<\cdots<x_{p-1}<x$ is an arithmetic progression.
\end{defn}

\begin{prop}
\label{thm:alt}
The $p$-Stanley sequence $S_p(A)$ is the unique sequence that starts with $A$, is $p$-free, and $p$-covers all $x\not\in S_p(A)$ with $x>\max(A)$.
\end{prop}

\begin{proof}
Since $x>\max(A)$ there are two cases. If $x$ is in $S_p(A)$, its addition to the sequence preserves that the sequence is $p$-free. If $x$ is not in $S_p(A)$, it follows that the addition of $x$ would have created a $p$-term arithmetic progression with largest term $x$ and with the remaining terms in $S_p(A)$.
\end{proof}

\begin{defn}
A set $A\subseteq\{0,1,\ldots,N-1\}$ is said to \textit{$p$-cover $x$ mod $N$} if there exist $x_1,x_2,\ldots,x_{p-1}\in A$ such that $x_1<x_2<\cdots<x_{p-1}$ and $x$ form an arithmetic progression mod $ N$. Restricting $0\le x<N$ and given the size restrictions for $A$ this is equivalent to $x_1<x_2<\cdots<x_{p-1}<x$ or $x_1<x_2<\cdots<x_{p-1}<x+N$ forming an arithmetic progression.
\end{defn}

\begin{defn}
A set $A\subseteq\{0,1,\ldots,N-1\}$ is a \textit{modular $p$-free set mod $N$} if $A$ contains 0, is $p$-free mod $N$, and $p$-covers all $x$ with $0\leq x<N$ and $x\not\in A$. A $p$-Stanley sequence is a \textit{modular $p$-Stanley sequence} if it has the form $S_p(A)$ for a modular $p$-free set $A$.
\end{defn}

We will refer to ``$p$-covering" and ``modular $p$-free" simply as ``covering" and ``modular" when $p$ is obvious. We write $A+B$ for $\{a+b\mid a\in A, b\in B\}$ and $c\cdot A$ for $\{c\cdot a\mid a\in A\}$. The following is the main theorem on modular $p$-Stanley sequences proved in \cite{mod}. It implies that a modular Stanley sequence grows asymptotically as $S_p(0)$. 

\begin{thm}[Theorem 6.5 in \cite{mod}]
If $A$ is a modular $p$-free set mod $N$, then $S_p(A)=A+N\cdot S_p(0)$. Note that $S_p(0)$ consists of all nonnegative integers with no $p-1$ in their base $p$ expansions.
\end{thm}

\begin{cor}[Corollary 6.6 in \cite{mod}]
Any modular $p$-Stanley sequence exhibits Type 1 growth.
\end{cor}

\section{First Class of $p$-Stanley sequences}
\label{sec:const1}

We use the notation $t_i(x)$ to refer to the digit corresponding to $p^i$ in the base $p$ expansion of $x$. We initially define a pair of sets which are critical for this section.

\begin{defn}
Let $A_p^k$ be the set of positive integers $n$ such that $2\cdot p^{k-1}<n\leq p^k$ with $p^k-n\in S_p(0)$. This is equivalent to $t_i(p^k-n)\neq p-1$ for all $i$ and additionally $t_{k-1}(p^k-n)\neq p-2$. Let $A_p=\bigcup\limits_{k=0}^{\infty}A_p^{k}$.
\end{defn}
For example the set $A_5$ begins $\{1,3,4,5,12,13,14,15,17,18\ldots\}$.
\begin{nota}
Let $S_p^k=\{x\mid x\in S_p(0),x<p^k\}$. Note by Lemma 6.4 in \cite{mod}, $S_p^k$ is $p$-free mod $p^k$ and covers $\{0,1,\ldots,p^k-1\}\setminus S_p^k$.
\end{nota}
In a manner closely related to the proof of Lemma 6.4 in \cite{mod}, we define a key procedure for the proof of Theorem 3.4.

\begin{defn}
For $0\leq x<p^k$ define the \textit{canonical covering} of $x$ to be the sequence $x_1,x_1,\ldots,x_{p-1}$ where $x_j=\sum_i t_i^{(j)}p^i$ and $t_i^{(j)}=t_i(x)$ if $t_i(x)\neq p-1$ and $t_i^{j}=j-1$ if $t_i(x)= p-1$.
\end{defn}

Note that the canonical covering is contained in $S_p^k$ and, as suggested by its name, $p$-covers $x$. Using these definitions it possible to prove our first result on modular $p$-Stanley sequences. 

\begin{thm}
For $p>3$ a prime and $n\in A_p$, $S_p(0,n)$ is a modular p-Stanley sequence.
\end{thm}

\begin{proof}
Suppose that $k$ is such that $p^{k-2}<n\leq p^{k-1}$, and let $A=\{0\}\cup(n+S_p^k)\setminus\{p^{k-1}(p-1)\}$. Note that $\max(A)<p^k$. Therefore it suffices to demonstrate $S_p(0,n)=S_p(A)$ and that $A$ is modular mod $p^k$. 

To demonstrate that $S_p(0,n)=S_p(A)$, it suffices by Proposition \ref{thm:alt} to prove that $A$ is $p$-free and covers all $n<x<p^k$ with $x\not\in A$. To demonstrate that $A$ is modular mod $ p^k$, it suffices to prove that $A$ is $p$-free mod $p^k$ and covers all $0\leq x<p^k$ mod $ p^k$ with $x\not\in A$. Thus it is sufficient to show the slightly stronger statement that $A$ is $p$-free mod $ p^k$ and covers all $n<x<p^k +n$ with $x\not\in A$ and $x\neq p^k$. Let $A'=-n+A=\{-n\}\cup S_p^k\setminus\{p^{k-1}(p-1)-n\}$. We demonstrate that $A'$ has no arithmetic progressions mod $p^k$ which will give us the first of our two desired results.

Since $S_p^k$ is $p$-free mod $p^k$, any arithmetic progression in $A'$ must contain $-n$. Suppose there is an arithmetic progression $\{a_i\}$ mod $p^k$ and define $b_i\equiv a_i\mod{p^{k-1}}$ with $0\le b_i<p^{k-1}$. It follows that $\{b_i\}$ is an arithmetic progression mod $p^{k-1}$. By the definition of $A_p$, we know that $p^{k-1}-n\in S_p^k$, so the progression $\{b_i\}$ is in fact an arithmetic progression mod $p^{k-1}$ in $S_p^{k-1}$. Thus the progression $\{b_i\}$ must be the constant arithmetic progression. It follows that $a_0\equiv a_1\equiv\cdots\equiv a_{p-1}\equiv -n \pmod{p^{k-1}}$ and therefore the only possible arithmetic progression mod $p^k$ in $A'$ is $i\cdot p^{k-1}-n$ for $0\leq i <p$. However, since $(p-1)p^{k-1}-n\not\in A'$, it follows that $A'$ is $p$-free mod $p^k$.

To prove the second result we demonstrate that $A'$ covers $0<x<p^k$ with $x\not\in A'$ and $x\neq p^k-n$. If $x=p^{k-1}(p-1)-n$, then $x$ is covered by $\{ip^{k-1}-n\}$ for $0\leq i<p-1$. Otherwise, $x\not\in S_p^k$. Since $x$ is covered by its canonical covering in $S_p^k$, the only cases we have to consider are those in which the canonical covering of $x$ contains $p^{k-1}(p-1)-n$.

Let $m=p^{k-1}(p-1)-n$, since $n\in A_p$, we know that $t_{k-1}(m)=p-2$, $t_{k-2}(m)<p-2$, and $t_i(m)\neq p-1$ for all $i$. Any $0<x<p^k$ whose canonical covering contains $m$ can be written in the form \[x_S=\sum_{i=0\atop i\not\in S}^{k-1}t_i(m)p^i+\sum_{i\in S}(p-1)p^i,\] where $S\subseteq\{0,1,\ldots,k-1\}$ is a set of digits such that $t_i(m)$ is the same for all $i\in S$. We earlier assumed that $x\neq m$ and $x\neq p^k-n=p^{k-1}+m$. This implies that $S\neq\emptyset,\{k-1\}$.

For the remainder of the proof fix an integer $a$ and an $S\subseteq\{0,1,\ldots,k-1\}$ such that $a=t_i(m)$ for all $i\in S$ and $S\neq\emptyset,\{k-1\}$. Let $j$ be $\max(S\setminus\{k-1\})$ and let $b=t_{j+1}(m)$. 

We know that $t_{k-1}(m)=p-2$ and $t_{k-2}(m)<p-2$, which implies that $\{k-2,k-1\}\not\subseteq S$. Thus this implies that if $j=k-2$, then $k-1\not\in S$.

We know that $0\leq a,b<p-1$, and we now consider four cases.

\textbf{Case 1:} $a=0$.

Let $\Delta=\sum_{i\in S}p^i$. Then $\{p^{k-1}(p-1)-n+i\cdot\Delta\}$ for $0\leq i<p-1$ is the canonical covering of $x_S$ as we are preserving all digits not equals to $p-1$ in $x_S$ and using $\{0,\ldots,p-2\}$ where $x_S$ has a digit $p-1$. However $\{i\cdot p^{k-1}-n+i\cdot\Delta\}$ for $0\leq i<p-1$ also covers $x_S$. 

We need to check that all of these terms are in $A'$. Since $p^{k-1}(p-1)-n+i\Delta\in S_p^k$ with first digit $p-2$, then $i\cdot p^{k-1}-n+i\cdot\Delta$ is identical except the first digit ranges from $0$ through $p-2$ for $0<i<p-1$ while for $i=0$ it follows as $i\cdot p^{k-1}-n+i\cdot\Delta=-n\in A'$.

\textbf{Case 2:} $0<a<p-1$ and $0\leq b<(p-3)/2$.

Let $j'>j$ be the smallest integer such that $t_{j'}(m)\geq(p-1)/2$. Note $j'$ exists since $t_{k-1}(m)=p-2\geq (p-1)/2$. In this case take
\begin{align*}
\Delta&=\sum_{i=j}^{j'-1}p^i(p-1)/2+\sum_{i\in S\setminus\{j,j+1,\ldots,j'\}}p^i,\\
&=(p^{j'}-p^j)/2+\sum_{i\in S\setminus\{j,j+1,\ldots,j'\}}p^i
\end{align*}
and consider the arithmetic progression $\{x_S-i\cdot\Delta\}$ for $0<i\leq p-1$. We claim this set is contained in $A'$.

We can compute the digits of each of these numbers. Write the digit expansion of $x_S-i\cdot\Delta$ as $x_S-i\cdot\Delta=\sum_l t_l^{(i)}p^l$. For $l\not\in\{j,j+1,\ldots,j'\}$, then $t_l^{(i)}$ matches the canonical covering. In particular, $t_l^{(i)}=t_l(m)$ if $i\not\in S$ and otherwise $t_l^{(i)}=p-1-i$.

Using explicit computation it is possible to determine the remaining digits. First note that $t_{j'}^{(i)}=t_{j'}-\lceil i/2\rceil$. For $j+1<l<j'$, we have $t_{l}^{(i)}=t_l(m)$ for $i$ even and $t_l^{(i)}=t_l(m)+(p-1)/2$ for $i$ odd. Furthermore, $t_{j+1}^{(i)}=t_l(m)+1$ for $i>0$ even and $t_{j+1}^{(i)}=t_{j+1}(m)+1+(p-1)/2$ for $i$ odd. Finally, $t_j^{(i)}=i/2-1$ for $i>0$ even and $t_j^{(i)}=(p-1)/2+(i-1)/2$ for $i$ odd. 

Now we check that all of these terms are in $A'$. The $j$th digit cycles through each value when $0\le i\le p-1$, and since it equals $p-1$ when $i=0$, it never equals $p-1$ in the range $0<i\le p-1$ that we are using to cover $x_S$. Since $t_{j'}(m)\geq(p-1)/2$, $t_{j'}^{(i)}$ never goes below 0, and $t_{j'}^{(i)}<t_{j'}(m)$. Therefore we have $t_{j'}^{(i)}<p-1$ for $i>0$. Furthermore since $t_l(m)<(p-1)$/2 for $j<l<j'$, neither of the two values that this digit takes is $p-1$. Furthermore the $(j+1)$st digit only takes on 3 values, none of which is $p-1$ since $t_{j+1}(m)=b<(p-3)/2$. Finally, $t_{j+1}^{(i)}\neq t_{j+1}(m)$ for $i>0$. Since $t_{j+1}(m)$ never takes on its original value again, none of the terms in this sequence are $m$.

\textbf{Case 3:} $0<a<p-1$ and $(p-3)/2\leq b<p-1$ and $(a,b,p)\neq(2,1,5)$.

We claim we can find $1\leq d\leq b+1$ such that $d\nmid p-a-1$ given the conditions in this case. If $p>5$ it is not hard to check\footnote{Let $\prod_i p_i^{e_i}$ be the prime factorization of $p-1$. If $p-1$ is not a prime power, then $p_i^{e_i}\in\{1,\ldots,(p-1)/2\}$ for all $i$. Otherwise, since $p$ is odd, we can write $p-1=2^k$. Then since $k>2$,  $2^{k-1}$ and $3$ are elements in $\{1,2,\ldots,(p-1)/2\}$ and thus the least common multiple is at least $3\cdot2^{k-1}\geq 2^k=p-1$.} that $\lcm(1,2,\ldots,(p-1)/2)\geq p-1$, so a number in this range must not divide $p-a-1<p-1$. If $p=5$, we can use $d=2$ unless $a=2$ (and therefore $p-a-1=2$). Furthermore if $p=5$, $a=2$, $b\geq 2$, we can use $d=3$.

Let \[\Delta=d\cdot p^j+\sum_{i\in S\setminus\{j\}}p^i.\] We claim that the arithmetic progression $\{x_S-i\cdot\Delta\}$ for $0<i\leq p-1$ is contained in $A'$.

None of the digits of $x_S-i\cdot \Delta$ is equal to $p-1$ except for possibly the $j$th and $(j+1)$st digits. The $j$th digit decreases by $d\pmod{p}$ so it only takes on the value $p-1$ when $i=0$. Moreover, subtracting $\Delta$, the $j$th digit forces the $(j+1)$st to decrement exactly $d-1$ times (due to a ``borrow"). Since $p-1>b\geq(p-3)/2\geq d-1$, the $(j+1)$st digit never takes on the value $p-1$ and never itself ``borrows'' from the $(j+2)$nd digit.

Thus it suffices to check that no term is equal to $p^{k-1}(p-1)-n$. This must occur before the $(j+1)$st digit has changed its value from $t_{j+1}(m)$. In this range, the $j$th digit has value $t_j(x_S)-i\cdot d=(p-1)-i\cdot d$. However if $(p-1)-i\cdot d=a$, then $d\mid p-a-1$, a contradiction. Thus this arithmetic progression is contained in $A'$, as desired.

\textbf{Case 4:} $a=2$, $b=1$, and $p=5$.

This special case is similar to Case 2. Note that for $j<j'<k$, it is not the case that $j'\in S$. In particular the only possibility is $j'=k-1$, but $\{j,k-1\}\subseteq S$ implies that $t_{j}(m)=t_{k-1}(m)$ and $t_j(m)=a=2$ whereas $t_{k-1}(m)=p-2=3$. Furthermore note that $j+1\neq k-1$ since $t_{k-1}(m)=3\neq1=t_{j+1}(m)$. Now if $t_{j+2}(m)\geq 1$, letting \[\Delta=5^{j+1}+3\cdot5^j+\sum_{i\in S\setminus\{j\}}5^i,\]
it is easy to check that $\{x_S-i\cdot\Delta\}$ for $0<i\leq 4$ is in $A'$.

Otherwise, $t_{j+2}(m)=0$. Let $j'>j+2$ be the smallest integer such that $t_{j'}(m)\geq 2$. This exists for the same reason as in Case 2.
Now let 
\begin{align*}
\Delta&=\left(\sum_{i=j+2}^{j'-1}2\cdot5^i\right)+5^{j+1}+3\cdot5^j+\sum_{i\in S\setminus\{j,j+1,\ldots,j'\}}5^i,\\
&=(5^{j'}-5^{j+2})/2+5^{j+1}+3\cdot5^j+\sum_{i\in S\setminus\{j,j+1,\ldots,j'\}}5^i.
\end{align*}
We cover $x_S$ by $\{x_S-i\cdot\Delta\}$ for $0<i\leq 4$. By exactly the same reasoning as in Case 2, this covering is in $A'$.
\end{proof}

We conjecture, but cannot currently prove, that these are the only integers $n$ such that
$S_5(0, n)$ exhibits Type 1 growth. Computational evidence provided by
Moy and Rolnick \cite{mod} suggests that the integers less than 100 such that $S_5(0; n)$ are well-behaved and in particular modular are as follows:
\[1, 3, 4, 5, 12, 13, 14, 15, 17, 18, 19, 20, 22, 23, 24, 25, 37, 39, 40, 42, 43, 44, 45, 47, 57, 58, 59,
60,\]\[ 62, 63, 64, 65, 67, 68, 69, 70, 72, 73, 74, 75, 82, 83, 84, 85, 87, 88, 89, 90, 92, 93, 94, 95, 97, 98, 99.\]
See Problem 6.7 in \cite{mod} for more detail. This matches exactly the integers which Theorem 3.4 would suggest,
giving some support for this conjecture

\section{Second Construction of $p$-Stanley Sequences}
\label{sec:const2}

This section presents a generalization of Theorem 1.2 given by Rolnick \cite{rol1} with a proof that is similar in spirit to that of Theorem 1.2. For this section, fix an odd prime $p$, and recall that $t_i(x)$ refers to the $i$th digit of $x$ in base $p$.

\begin{defn}
We say a (positive) integer $x$ \emph{dominates} an integer $y$ if $t_i(x)\geq t_i(y)$ for all integers $i$.
\end{defn}

Note that the set $S_p^k$ defined in Section \ref{sec:const1} is exactly the set of integers dominated by $\sum_{i=0}^{k-1}(p-2)p^i$.

\begin{thm}
Let $T\subseteq S_p^k$ be a nonempty set that is downward-closed under the domination ordering. Namely if $x\in T$ and $y$ is dominated by $x$, then $x\in T$. Then $S_p(T\cup \{p^k\})$ and $S_p(T\cup \{(p-1) p^k\})$ are modular $p$-Stanley sequences.
\end{thm}

Note that for $p=3$ this is Theorem 1.2 in Rolnick \cite{rol2}.

\begin{proof}
In both cases, we give an explicit description of the Stanley $p$-sequences and prove that this is the correct sequence.

We claim that $x\in S_p(T\cup \{p^k\})$ if and only if the following three conditions hold
\begin{itemize}
\item $t_i(x)\neq p-1$ for $i\neq k$,
\item $t_k(x)=0$ implies that $\sum_{i=0}^{k-1} t_i(x)p^i\in T$,
\item $t_k(x)=p-1$ implies that $\sum_{i=0}^{k-1} t_i(x)p^i\not\in T$.
\end{itemize}

For convenience let $L$ be the set of integers satisfying the above relations. Note that $L\cap\{0,1,\ldots,p^k\}=T\cup\{p^k\}$. It suffices by Proposition 2.2 to demonstrate that $L$ does not contain any $p$-term arithmetic progressions and that every integer not in $L$ and greater than $p^k$ is covered by a $p$-term arithmetic progression in $L$.

To show that $L$ is $p$-free we proceed by contradiction. Suppose that $x_1<\cdots<x_p$ form an arithmetic progression. Let $i$ be the smallest integer such that $t_i(x_1),\ldots, t_i(x_p)$ are not all equal. Since $p$ is prime and the first $i$ digits of $x_1,\ldots,x_p$ are the same, this implies that $\{t_i(x_1),\ldots,t_i(x_p)\}=\{0,\ldots,p-1\}$. Since $t_i(x)\neq p-1$ for $i\neq k$, we conclude that $i=k$.

Now there are some $j,j'$ such that $t_k(x_j)=0$ and $t_k(x_{j'})=p-1$. By the definition of $L$, this implies that $\sum_{i=0}^{k-1}t_i(x_j)p^i\in T$ and $\sum_{i=0}^{k-1}t_i(x_{j'})p^i\not\in T$. However, since $t_i(x_j)=t_i(x_{j'})$ for $i<k$, this is a contradiction.

It remains to show that every integer $x>p^{k}$ is covered by a $p$-term arithmetic progression. In order to do so we explicitly construct a $p$-term arithmetic progression $x_1\le x_2\le \cdots\le x_{p-1}\le x$ with the $x_i$ in $L$. If we have equality anywhere in this chain then $x$ in $L$; otherwise $x_1<x_2<\cdots<x_{p-1}<x$ as desired. For $0\le i\le k-1$ if $t_i(x)=\ell<p-1$, then set $t_i(x_j)=\ell$ for $1\le j\le p-1$. If instead $t_i(x)=p-1$, set $t_i(x_j)=j-1$ for $1\le j\le p-1$. Note that this is exactly the canonical covering from earlier. Now we subdivide into several possible cases.

\textbf{Case 1:} $t_k(x) \neq 0,~p-1$

Set $t_k(x_i)=\ell$. For the remaining digits, use the canonical covering as before.

\textbf{Case 2:} $t_k(x)=p-1$

We have two cases. If the last $k$ digits of $x_1$ are in $T$, then set $t_k(x_j)=j-1$. Otherwise set $t_k(x_j)=p-1$. In either case, use the canonical covering for the remaining digits.

\textbf{Case 3:} $t_k(x)=0$ 

If the last $k$ digits of $x_{p-1}$ are in $T$, set $t_k(x_j)=0$ and use the canonical covering for the remaining digits. Otherwise, set $t_k(x_j)=j$ and perform the canonical covering for $x-p^{k+1}$ for the remaining higher digits. (Note that since $x>p^k$ and $t_k(x)=0$ it follows that $x\ge p^{k+1}$.)

It is routine to verify in each case that the $x_j$ constructed are in $L$, completing the proof that $S_p(T\cup\{p^k\})=L$. To show that this is a modular Stanley sequence, let $L^{*}=\{x\mid x\in L,x<p^{k+1}\}$. We claim that $L^{*}$ is a modular set. The proof of this fact is nearly identical to the above analysis. Consider just the digits $t_i(x)$ for $0\leq i\leq k$.

Next we prove that $S(T\cup \{(p-1)p^k\})$ is a modular $p$-Stanley sequence. This proof is similar to the above argument though slightly more involved. We claim that $x\in S(T\cup\{(p-1)p^k\})$ if and only if the following four conditions hold
\begin{itemize}
\item $t_i(x)\neq p-1$ for $i\neq k, k+1$,
\item $t_k(x)\neq p-2$,
\item $t_{k+1}(x)=0$ implies that $t_k(x)=0$ and $\sum_{i=0}^{k-1} t_i(x)p^i\in T$ or $t_k(x)=p-1$,
\item $t_{k+1}(x)=p-1$ implies that $t_k(x)\neq p-2, p-1$, and if $t_k(x)=0$, then $\sum_{i=0}^{k-1} t_i(x)p^i\not \in T$.
\end{itemize}

Again let $L$ be the set defined by these four conditions. We show that $L$ is $p$-free and $p$-covers the part of its complement greater than $(p-1)p^k$.

For the sake of contradiction, suppose that $x_1<x_2<\cdots<x_p$ form an arithmetic progression with $x_i\in L$. Using the same idea as above we see that $t_i(x_1)=\cdots=t_i(x_p)$ for $0\le i\le k-1$. Since $t_k(x)\neq p-2$, it follows that $t_k(x_1)=\cdots=t_k(x_p)$. Now if $\{t_{k+1}(x_1),\ldots,t_{k+1}(x_p)\}=\{0,\ldots,p-1\}$, then there exist $j,j'$ such that $t_{k+1}(x_{j})=0$ and $t_{k+1}(x_{j'})=p-1$. Then we see that $\sum_{i=0}^{p-1}t_i(x_{j})p^i\in T$ and  $\sum_{i=0}^{p-1}t_i(x_{j'})p^i\not\in T$. Thus we conclude that $t_{k+1}(x_1)=\cdots=t_{k+1}(x_p)$, and by the same reasoning we see that $x_1=\ldots=x_p$, a contradiction. 

It remains to show that every integer $x>(p-1)p^{k}$ is covered by a $p$-term arithmetic progression. In order to do so, we explicitly construct a $p$-term arithmetic progression, $x_1\le x_2\le \cdots\le x_{p-1}\le x$ with $x_i\in L$. If we have equality anywhere in this chain then $x\in L$. Otherwise, $x_1<x_2<\cdots<x_{p-1}$ as desired. For $0\le i\le k-1$, if $t_i(x)=\ell<p-1$, then set $t_i(x_j)=\ell$ for $1\le j\le p-1$. Otherwise $t_i(x)=p-1$, and we set $t_i(x_j)=j-1$ for $1\le j\le p-1$. We will define this procedure as earlier to be the canonical covering. Now we subdivide into several possible cases and note that several of these cases degenerate when $p=3$.

\textbf{Case 1:}$t_{k+1}(x)=1,\ldots,p-2$ and $t_{k}(x)\neq p-2$ 

Set $t_{k+1}(x)=t_{k+1}(x_j)$ and $t_{k}(x)=t_{k}(x_j)$ for $1\le j\le p-1$. For the remaining digits, use the canonical covering.

\textbf{Case 2:} Either $t_{k+1}(x)=p-1$ and $t_{k}(x)=1,\ldots, p-3$ or $t_{k+1}(x)=0$ and $t_{k}(x)=p-1$

Set $t_{k+1}(x)=t_{k+1}(x_j)$ and $t_{k}(x)=t_{k}(x_j)$ for $1\le j\le p-1$. For the remaining digits, use the canonical covering as before.

\textbf{Case 3:} $t_{k+1}(x)=p-1$ and $t_{k}(x)=p-1$

Set $t_{k+1}(x_j)=j-1$ and $t_{k}(x_j)=j-1$ for $1\le j\le p-1$. For the remaining digits, use the canonical covering as before.

\textbf{Case 4:} $t_{k+1}(x)=1,\ldots, p-1$ and $t_{k}(x)=p-2$

Set $t_{k+1}(x_j)=t_{k+1}(x)$ and $t_{k+1}(x_j)=j-2$ for $2\le j\le p-1$ while $t_{k+1}(x_1)=t_{k+1}(x)-1$ and $t_{k}(x_1)=p-1$. For the remaining digits, use the canonical covering as before.

\textbf{Case 5:} $t_{k+1}(x)=0$ and $t_{k}(x)=1,\ldots,p-3$

Set $t_{k+1}(x_j)=j$ and $t_{k}(x_j)=t_k(x)$ for $1\le j\le p-1$. For the remaining digits, use the canonical covering $x-p^{k+2}$.

\textbf{Case 6:} $t_{k+1}(x)=0$ and $t_{k}(x)=p-2$

Set $t_{k+1}(x_j)=j$ and $t_{k}(x_j)=j-2$ for $2\le j\le p-1$. Also put $t_{k+1}(x_1)=0$ and $t_{k}(x_j)=p-1$. For the remaining digits, use the canonical covering $x-p^{k+2}$.

\textbf{Case 7:} $t_{k+1}(x)=t_k(x)=0$

Consider $x_{p-1}$ before setting $t_{k+1}(x_{p-1})$ and $t_k(x_{p-1})$. If $x_{p-1}\in L$, then set $t_{k+1}(x_j)=t_k(x_j)=0$ for $1\le j\le p-1$ and for the remaining digits, use the canonical covering $x$. Otherwise, set $t_{k+1}(x_j)=j$ and $t_k(x_j)=0$ for $1\le j\le p-1$ and use the canonical covering $x-p^{k+2}$ for the remaining digits.

\textbf{Case 8:} $t_{k+1}(x)=p-1$ and $t_k(x)=0$

Consider $x_{p-1}$ before setting $t_{k+1}(x_{p-1})$ and $t_k(x_{p-1})$. If $x_{p-1}\in L$, then set $t_{k+1}(x_j)=j-1$ and $t_k(x_j)=0$ for $1\le j\le p-1$ and for the remaining digits, use the canonical covering $x$. Otherwise set $t_{k+1}(x_j)=p-1$ and $t_k(x_j)=0$ for $1\le j\le p-1$  and for the remaining digits, use the canonical covering $x$.

In each case it is routine to verify in each case that the $x_j$ constructed are in $L$ and form an arithmetic progression with $x$ being the largest term. Finally, to show that this sequence is modular, let $L^{*}=\{x\mid x\in L, x<p^{k+2}\}$. We claim that $L^{*}$ is a modular set. Demonstrating that $L^{*}$ is modular is nearly identical to above analysis considering $t_i(x)$ for $0\leq i\leq k+1$ and is omitted.
\end{proof}

\section{Conclusions}
\label{sec:concl}

The two constructions in this paper are among the first classes of large modular $p$-Stanley sequences. These constructions raise several natural questions. The first follows naturally from the computational evidence in Section 3 and conjecturally answers a question of Moy and Rolnick \cite{mod} regarding which sets $\{0,n\}$ generate modular $p$-Stanley sequences.
\begin{conj}
The sequence $S_p(0,n)$ is a modular $p$-Stanley sequence if and only if $n\in A_p$.
\end{conj}
The next question deals with $p$-Stanley sequences generated in manners similar to that the second construction. 
\begin{ques}
Consider a set $S\subseteq \{1,\ldots, p^{k}-1\}$ and $1\le i\le p-2$. Under what conditions is $S_p(S\cup \{0,p^k,\ldots, i\cdot p^k\})$ a modular $p$-Stanley sequence?
\end{ques}
Finally, we end on another construction of $p$-Stanley sequences that appears to hold for small integers $x$ but for which an explicit characterization appears difficult. This is the natural analog of Lemma 3.5 in Rolnick \cite{rol2} and appears to suggest a further connection between the domination order and $p$-Stanley sequences.
\begin{conj}
Consider an integer $x$ with no $p-1$ in its base $p$ expansion. If $T$ is the set of all integers dominated by $x$, then $S_p(T)$ is a modular $p$-Stanley sequence.
\end{conj}
\section*{Acknowledgments}

This research was conducted at the University of Minnesota Duluth REU and was supported by NSF grant 1358695, NSF grant 1659047, and NSA grant H98230-13-1-0273. The authors wish to thank Joe Gallian for suggesting the problem and Colin Defant, Ben Gunby, Hannah Alpert, and Levent Alpoge for helpful comments on the manuscript.

\end{document}